\documentclass[12pt]{article}

\usepackage{amsmath, amssymb, amsthm}
\usepackage[margin=1in]{geometry}
\usepackage{array}
\usepackage{hyperref}
\newtheorem{theorem}{Theorem}
\newtheorem{lemma}{Lemma}
\newtheorem{corollary}{Corollary}
\title{Intersection Arrays of Completely Regular Codes of Covering Radius One in Generalized Petersen Graphs}
\author{
Hamed Karami\thanks{Corresponding author.}\\
Department of Mathematics and Statistics, Georgia State University\\
\texttt{hkarami@student.gsu.edu}
}
\date{}
\begin{document}

\maketitle

\begin{abstract}
We determine all possible intersection arrays of completely regular codes of
covering radius one in the generalized Petersen graphs \(GP(n,k)\), where
\(n\geq 3\) and \(1\leq k<n/2\). In the equivalent language of perfect
colorings, this amounts to enumerating all quotient matrices of perfect
\(2\)-colorings, up to interchanging the two colors. Since \(GP(n,k)\) is cubic,
there are only six possible nontrivial quotient matrices. For each of them, we
give necessary and sufficient arithmetic conditions on \(n\) and \(k\) for its
existence. The feasible cases are realized by explicit periodic colorings. The
nonexistence part is obtained by reducing the local coloring conditions to
cyclic systems of linear equations and applying a Fourier argument on roots of
unity. Together with the previously known cases \(GP(n,2)\) and \(GP(n,3)\),
the results give a complete arithmetic classification of quotient matrices, and
hence of covering-radius-one completely regular code parameters, in the
generalized Petersen family.
\end{abstract}



\section{Introduction}\label{sec:introduction}

Completely regular codes were introduced by Delsarte in the study of
association schemes and coding theory. Let \(G\) be a connected graph and let
\(C\subseteq V(G)\). The distance partition of \(V(G)\) with respect to \(C\)
is called equitable if every vertex in one part has a fixed number of
neighbours in every other part. In this case \(C\) is called a completely
regular code. Thus, when the covering radius is one, \(C\) and its complement
form a perfect \(2\)-coloring of \(G\). Conversely, in a connected regular
graph, every nontrivial perfect \(2\)-coloring gives completely regular codes
of covering radius one, one for each color class. Hence the problem considered
in this paper may be stated either in the language of completely regular codes
or in the language of perfect colorings.

We shall use the latter language in the proofs. A perfect coloring is also
called an equitable partition or a partition design. If
\(T:V(G)\to\{1,\ldots,m\}\) is a perfect coloring, then there is a matrix
\(S=(s_{ij})\) such that every vertex of color \(i\) has exactly \(s_{ij}\)
neighbours of color \(j\). The matrix \(S\) is called the quotient matrix
(or the parameter matrix) of the coloring. The enumeration of such matrices is
a standard way to study completely regular codes in families of regular graphs.

Perfect colorings and completely regular codes have been considered in many
classical graphs. For Hamming graphs and hypercubes see
\cite{FonDerFlaass2007Hypercube,Potapov2012,
BespalovKrotovMatiushevTaranenkoVorobev2021}; for Johnson graphs see
\cite{AvgustinovichMogilnykh2008,AvgustinovichMogilnykh2011,
Mogilnykh2007,GavrilyukGoryainov2013}. They are also closely related to
orthogonal arrays, correlation-immune functions, bent functions, and
combinatorial designs; see
\cite{Delsarte1973,Krotov2011,PotapovAvgustinovich2020,KrotovPotapov2025}.
General facts on equitable partitions and completely regular codes may be
found in
\cite{Neumaier1992,BrouwerCohenNeumaier1989,GodsilRoyle2001,
BorgesRifaZinoviev2019,KrotovPotapov2025}.

In this paper we consider the generalized Petersen graph \(GP(n,k)\). This
graph has vertex set
\[
   V(GP(n,k)) = \{a_i,b_i: i\in \mathbb Z_n\}
\]
and edge set
\[
   E(GP(n,k)) = \{a_i a_{i+1},\, a_i b_i,\, b_i b_{i+k}: i\in \mathbb Z_n\},
\]
where the indices are read modulo \(n\). We assume throughout that
\(n\geq 3\) and \(1\leq k<n/2\). The graph \(GP(n,k)\) is connected, since the outer vertices form an
\(n\)-cycle and each inner vertex \(b_i\) is adjacent to \(a_i\). The graphs \(GP(n,k)\) form a well-known
family of cubic graphs. They go back to Coxeter and Watkins, and their
automorphism groups, Hamilton cycles, and isomorphism classes have been
studied in
\cite{Coxeter1950,Watkins1969,FruchtGraverWatkins1971,
Alspach1983,SteimleStaton2009}.

Perfect colorings of generalized Petersen graphs were first considered in
some special cases. The perfect \(2\)-colorings of \(GP(n,2)\) were enumerated
in \cite{AlaeiyanKarami2016GP}, and the case \(GP(n,3)\) was treated in
\cite{Karami2022GPn3}. Some related perfect colorings of small generalized
Petersen graphs and other cubic graphs were considered in
\cite{AlaeiyanKaramiSiasat2018GP3,AlaeiyanKarami2017Platonic}. The purpose of
the present paper is to complete the list for arbitrary \(k\).

Since \(GP(n,k)\) is cubic, every quotient matrix of a perfect \(2\)-coloring
has row sums equal to \(3\). Up to interchanging the two colors, there are only
six possible nontrivial matrices. We determine, for each of them, necessary
and sufficient arithmetic conditions on \(n\) and \(k\). The existence part is
given by explicit periodic colorings. For the nonexistence part, the local
conditions imposed by the quotient matrix are transformed into cyclic systems
of linear equations. A Fourier argument on the cyclic group then gives the
needed divisibility restrictions.

The paper is organized as follows. In Section~\ref{sec:preliminaries}, we
give the definitions and list the six possible matrices. Section~\ref{sec:A2}
is devoted to the matrix \(A_2\). Sections~\ref{sec:A3} and~\ref{sec:A4}
deal with \(A_3\) and \(A_4\). Section~\ref{sec:A5} treats \(A_5\). The
known matrices \(A_1\) and \(A_6\), together with the results proved here,
give the complete classification.

\section{Preliminaries and main results}\label{sec:preliminaries}

 Since \(GP(n,k)\) is cubic, every quotient matrix of a perfect
\(2\)-coloring has row sums equal to \(3\). Since the coloring is nontrivial and \(GP(n,k)\) is connected, both
off-diagonal entries of the quotient matrix are positive. Thus the unordered
pair of off-diagonal entries is one of
\[
   \{1,1\},\{1,2\},\{1,3\},\{2,2\},\{2,3\},\{3,3\},
\]
which gives the following six matrices.
\begin{equation}\label{eq:six-matrices}
\begin{array}{lll}
A_1=\begin{pmatrix}2&1\\[2pt]1&2\end{pmatrix},&
A_2=\begin{pmatrix}2&1\\[2pt]2&1\end{pmatrix},&
A_3=\begin{pmatrix}1&2\\[2pt]2&1\end{pmatrix},\\[12pt]
A_4=\begin{pmatrix}0&3\\[2pt]1&2\end{pmatrix},&
A_5=\begin{pmatrix}0&3\\[2pt]2&1\end{pmatrix},&
A_6=\begin{pmatrix}0&3\\[2pt]3&0\end{pmatrix}.
\end{array}
\end{equation}

We shall use the following elementary cases. The matrix \(A_1\) occurs for
every \(GP(n,k)\): color all vertices \(a_i\) with one color and all vertices
\(b_i\) with the other color. The matrix \(A_6\) occurs precisely when
\(GP(n,k)\) is bipartite. This is equivalent to \(n\) being even and \(k\) being
odd. These two cases also follow from the general observations in
\cite{AlaeiyanKarami2016GP}.

\begin{theorem}[Main classification]\label{thm:main-classification}
Let \(n\geq 3\) and \(1\leq k<n/2\). With \(A_1,\ldots,A_6\) as in
\eqref{eq:six-matrices}, the quotient matrices of perfect \(2\)-colorings of
\(GP(n,k)\), up to interchanging the two colors, occur exactly as follows:
\begin{enumerate}
\item \(A_1\) occurs for all \(n,k\) satisfying the above assumptions;
\item \(A_2\) occurs if and only if \(3\mid n\) and \(3\nmid k\);
\item \(A_3\) occurs if and only if \(2\mid n\) and \(k\) is odd;
\item \(A_4\) occurs if and only if \(4\mid n\) and \(k\) is odd;
\item \(A_5\) occurs if and only if \(5\mid n\) and \(k\equiv \pm2\pmod 5\);
\item \(A_6\) occurs if and only if \(2\mid n\) and \(k\) is odd.
\end{enumerate}
\end{theorem}

\begin{corollary}\label{cor:cr-codes}
Let \(C\) be a nonempty proper subset of \(V(GP(n,k))\). Then \(C\) is a
completely regular code of covering radius one if and only if the partition
\(\{C,V(GP(n,k))\setminus C\}\) has one of the quotient matrices listed in
Theorem~\ref{thm:main-classification}, up to interchanging the two cells. Equivalently, the possible intersection arrays \(\{b;c\}\) of such codes are
as follows:
\[
\begin{array}{c|c}
\text{intersection array} & \text{condition}\\ \hline
\{1;1\} & \text{all } n,k\\
\{1;2\},\{2;1\} & 3\mid n,\; 3\nmid k\\
\{2;2\} & 2\mid n,\; k \text{ odd}\\
\{3;1\},\{1;3\} & 4\mid n,\; k \text{ odd}\\
\{3;2\},\{2;3\} & 5\mid n,\; k\equiv\pm2\pmod 5\\
\{3;3\} & 2\mid n,\; k \text{ odd}.
\end{array}
\]
\end{corollary}

\begin{proof}
If \(C\) is a completely regular code of covering radius one, then its distance
partition is
\[
   \{C,V(GP(n,k))\setminus C\}.
\]
By definition this partition is equitable, and hence gives a perfect
\(2\)-coloring. Conversely, any perfect \(2\)-coloring with two nonempty color
classes gives an equitable partition of the form
\[
   \{C,V(GP(n,k))\setminus C\}.
\]
Since each quotient matrix listed in Theorem~\ref{thm:main-classification} has
positive off-diagonal entries, every vertex outside \(C\) has a neighbour in
\(C\). Thus the covering radius is one.

\noindent For a two-cell equitable partition with quotient matrix
\[
   \begin{pmatrix}
      a & b\\
      c & d
   \end{pmatrix},
\]
the corresponding intersection array of the first cell is \(\{b;c\}\). Reading
the off-diagonal entries of the matrices in
Theorem~\ref{thm:main-classification}, and allowing the two cells to be
interchanged, gives exactly the table displayed above.
\end{proof}

\begin{lemma}[A circulant system]\label{lem:circulant-system}
Let \(q\geq 2\), and let \(\alpha\) and \(k\) be integers. Let
\((U_r)_{r\in\mathbb Z_q}\) and \((V_r)_{r\in\mathbb Z_q}\) be two sequences of
complex numbers satisfying
\[
   U_{r-1}+\alpha U_r+U_{r+1}+V_r=0,
   \qquad r\in\mathbb Z_q,
\]
and
\[
   U_r+V_{r-k}+\alpha V_r+V_{r+k}=0,
   \qquad r\in\mathbb Z_q.
\]
Assume that, for every \(q\)-th root of unity \(\xi\), $(\alpha+\xi+\xi^{-1})(\alpha+\xi^k+\xi^{-k})\neq 1$.
Then, we have $U_r=V_r=0$, for all $r\in\mathbb Z_q$. 
\end{lemma}

\begin{proof}
Let $\omega=e^{2\pi i/q}$. For \(j\in\mathbb Z_q\), put
\[
   \widehat U(j)=\sum_{r\in\mathbb Z_q} U_r\omega^{-jr},
   \qquad
   \widehat V(j)=\sum_{r\in\mathbb Z_q} V_r\omega^{-jr}.
\]
Multiplying the first equation by \(\omega^{-jr}\) and summing over
\(r\in\mathbb Z_q\), we get
\[
   (\alpha+\omega^j+\omega^{-j})\widehat U(j)+\widehat V(j)=0.
\]
Similarly, the second equation gives
\[
   \widehat U(j)+
   (\alpha+\omega^{jk}+\omega^{-jk})\widehat V(j)=0.
\]
Thus, for every \(j\in\mathbb Z_q\), the pair
\((\widehat U(j),\widehat V(j))\) is a solution of the linear system
\[
\begin{pmatrix}
\alpha+\omega^j+\omega^{-j} & 1\\
1 & \alpha+\omega^{jk}+\omega^{-jk}
\end{pmatrix}
\binom{\widehat U(j)}{\widehat V(j)}
=
\binom{0}{0}.
\]
By the assumption, the determinant of this matrix is nonzero. Hence
\[
   \widehat U(j)=\widehat V(j)=0,
   \qquad j\in\mathbb Z_q.
\]
Since the discrete Fourier transform on \(\mathbb Z_q\) is invertible, it follows that
\[
   U_r=V_r=0,
   \qquad r\in\mathbb Z_q.
\]
\end{proof}
\begin{lemma}[Non-vanishing of the determinants]
\label{lem:root-nonzero}
The determinant condition in Lemma~\ref{lem:circulant-system} holds in the
following cases.
\begin{enumerate}
\item If \(q=2^s\), \(s\geq 1\), \(2\mid k\), and \(\xi^q=1\), then
\[
   (1+\xi+\xi^{-1})(1+\xi^k+\xi^{-k})\neq 1.
\]

\item If \(q=5^s\), \(s\geq 1\), \(k\equiv 0,\pm1\pmod 5\), and
\(\xi^q=1\), then
\[
   (2+\xi+\xi^{-1})(2+\xi^k+\xi^{-k})\neq 1.
\]

\item If \(q=3^s\), \(s\geq 1\), \(3\mid k\), and \(\xi^q=1\), then
\[
   (\xi+\xi^{-1})(\xi^k+\xi^{-k})\neq 1.
\]
\end{enumerate}
\end{lemma}

\begin{proof}
We shall use the following elementary fact. Let \(p\) be a prime and let
\(t\geq 1\). Denote by \(\Phi_{p^t}(x)\) the \(p^t\)-th cyclotomic polynomial,
and let \(\varphi\) be Euler's totient function. Recall that
\[
   \varphi(p^t)=p^{t-1}(p-1).
\]
If \(\zeta\) is a primitive \(p^t\)-th root of unity and
\(F(x)\in\mathbb Z[x]\) satisfies \(F(\zeta)=0\), then
\[
   \Phi_{p^t}(x)\mid F(x).
\]
Indeed, \(\Phi_{p^t}(x)\) is the minimal polynomial of \(\zeta\) over
\(\mathbb Q\). Moreover,
\[
   \Phi_{p^t}(x)=\frac{x^{p^t}-1}{x^{p^{t-1}}-1},
\]
and hence, after reducing modulo \(p\),
\[
   \Phi_{p^t}(x)\equiv (x-1)^{p^t-p^{t-1}}
   =(x-1)^{\varphi(p^t)} \pmod p.
\]
Consequently, if \(F(\zeta)=0\), then
\[
   (x-1)^{\varphi(p^t)}\mid \overline{F}(x)
   \quad \text{in } \mathbb F_p[x],
\]
where \(\overline{F}\) denotes the reduction of \(F\) modulo \(p\).

\noindent We prove \((1)\). Suppose, to the contrary, that
\[
   (1+\xi+\xi^{-1})(1+\xi^k+\xi^{-k})=1.
\]
Multiplying by \(\xi^{k+1}\), we obtain \(F_2(\xi)=0\), where
\[
   F_2(x)=(x^2+x+1)(x^{2k}+x^k+1)-x^{k+1}.
\]
The roots of orders \(1\), \(2\), and \(4\) give no solution. Indeed, for
\(\xi=1\) the left-hand side is \(9\); for \(\xi=-1\), since \(k\) is even, it
is \(-3\); and if \(\xi^2=-1\), then \(\xi+\xi^{-1}=0\), and the left-hand
side is either \(3\) or \(-1\). Thus \(\xi\) has order \(2^t\) with \(t\geq 3\). Since \(2\mid k\), a direct expansion at \(x=1\) gives
\[
   F_2(1+z)\equiv z^2+z^3h_2(z) \pmod 2
\]
for some \(h_2(z)\in\mathbb F_2[z]\). Equivalently,
\[
   \overline{F_2}(1+z)=z^2+z^3h_2(z)
   \quad \text{in } \mathbb F_2[z].
\]
Hence the reduction \(\overline{F_2}\in\mathbb F_2[x]\) has \(x=1\) as a root
of multiplicity exactly \(2\). On the other hand,
\(\Phi_{2^t}(x)\mid F_2(x)\) would imply that
\[
   (x-1)^{\varphi(2^t)}=(x-1)^{2^{t-1}}
\]
divides \(\overline{F_2}(x)\) in \(\mathbb F_2[x]\). Since \(t\geq 3\), we have
\(2^{t-1}\geq 4\), a contradiction. This proves \((1)\).

\noindent We prove \((2)\). Suppose that
\[
   (2+\xi+\xi^{-1})(2+\xi^k+\xi^{-k})=1.
\]
Multiplying by \(\xi^{k+1}\), we get \(F_5(\xi)=0\), where
\[
   F_5(x)=(x+1)^2(x^k+1)^2-x^{k+1}.
\]
For \(\xi=1\), the left-hand side is \(16\), not \(1\). Hence \(\xi\) has
order \(5^t\) for some \(t\geq 1\). Expanding at \(x=1\), we obtain
\[
   F_5(1+z)\equiv (4-k^2)z^2+z^3h_5(z) \pmod 5
\]
for some \(h_5(z)\in\mathbb F_5[z]\). Equivalently,
\[
   \overline{F_5}(1+z)=(4-k^2)z^2+z^3h_5(z)
   \quad \text{in } \mathbb F_5[z].
\]
Since \(k\equiv 0,\pm1\pmod 5\), we have
\[
   4-k^2\not\equiv 0\pmod 5.
\]
Thus the reduction \(\overline{F_5}\in\mathbb F_5[x]\) has \(x=1\) as a root
of multiplicity exactly \(2\). But \(\Phi_{5^t}(x)\mid F_5(x)\) would imply
that
\[
   (x-1)^{\varphi(5^t)}
\]
divides \(\overline{F_5}(x)\) in \(\mathbb F_5[x]\). Since
\(\varphi(5^t)=4\cdot 5^{t-1}\geq 4\), this is impossible. Therefore \((2)\) holds.

\noindent Finally, we prove \((3)\). Suppose that
\[
   (\xi+\xi^{-1})(\xi^k+\xi^{-k})=1.
\]
Multiplying by \(\xi^{k+1}\), we get \(F_3(\xi)=0\), where
\[
   F_3(x)=(x^2+1)(x^{2k}+1)-x^{k+1}.
\]
For \(\xi=1\), the left-hand side is \(4\), not \(1\). If \(\xi\) has order
\(3\), then $\xi+\xi^{-1}=-1.$
Moreover, since \(3\mid k\), we have \(\xi^k=\xi^{-k}=1\), and so
\[
   (\xi+\xi^{-1})(\xi^k+\xi^{-k})=-2\neq 1.
\]
Thus \(\xi\) has order \(3^t\) with \(t\geq 2\). Since \(3\mid k\), expansion at \(x=1\) gives
\[
   F_3(1+z)\equiv 2z^2+z^3h_3(z) \pmod 3
\]
for some \(h_3(z)\in\mathbb F_3[z]\). Equivalently,
\[
   \overline{F_3}(1+z)=2z^2+z^3h_3(z)
   \quad \text{in } \mathbb F_3[z].
\]
Therefore the reduction \(\overline{F_3}\in\mathbb F_3[x]\) has \(x=1\) as a
root of multiplicity exactly \(2\). But \(\Phi_{3^t}(x)\mid F_3(x)\) would
imply that
\[
   (x-1)^{\varphi(3^t)}
\]
divides \(\overline{F_3}(x)\) in \(\mathbb F_3[x]\). Since \(t\geq 2\), we have
\(\varphi(3^t)=2\cdot 3^{t-1}\geq 6\), a contradiction. This proves \((3)\).
\end{proof}
\begin{corollary}[Uniformity of cyclic systems]
\label{cor:cyclic-uniformity}
All subscripts in this corollary are taken modulo \(q\).

\begin{enumerate}
\item Let \(q=2^s\), \(s\geq 1\), and let \(k\) be even. Suppose that
\((X_r)_{r\in\mathbb Z_q}\) and \((Y_r)_{r\in\mathbb Z_q}\) are two sequences
of real numbers satisfying
\[
   X_{r-1}+X_r+X_{r+1}+Y_r=\lambda,
   \qquad r\in\mathbb Z_q,
\]
and
\[
   X_r+Y_{r-k}+Y_r+Y_{r+k}=\lambda,
   \qquad r\in\mathbb Z_q.
\]
Then
\[
   X_r=Y_r=\frac{\lambda}{4},
   \qquad r\in\mathbb Z_q.
\]

\item Let \(q=5^s\), \(s\geq 1\), and let \(k\equiv 0,\pm1\pmod 5\).
Suppose that
\((X_r)_{r\in\mathbb Z_q}\) and \((Y_r)_{r\in\mathbb Z_q}\) are two sequences
of real numbers satisfying
\[
   X_{r-1}+2X_r+X_{r+1}+Y_r=\lambda,
   \qquad r\in\mathbb Z_q,
\]
and
\[
   X_r+Y_{r-k}+2Y_r+Y_{r+k}=\lambda,
   \qquad r\in\mathbb Z_q.
\]
Then
\[
   X_r=Y_r=\frac{\lambda}{5},
   \qquad r\in\mathbb Z_q.
\]

\item Let \(q=3^s\), \(s\geq 1\), and let \(k\) be divisible by \(3\).
Suppose that
\((S_r)_{r\in\mathbb Z_q}\) is a sequence of real numbers satisfying
\[
   S_{r-k-1}+S_{r-k+1}+S_{r+k-1}+S_{r+k+1}
   =
   S_r+\lambda,
   \qquad r\in\mathbb Z_q.
\]
Then
\[
   S_r=\frac{\lambda}{3},
   \qquad r\in\mathbb Z_q.
\]
\end{enumerate}

In particular, if the sequences are integer-valued and \(\lambda\) is an
integer, then \(\lambda\) is divisible by \(4\), \(5\), and \(3\) in
\((1)\), \((2)\), and \((3)\), respectively.
\end{corollary}

\begin{proof}
We prove the three parts separately. For \((1)\), put
\[
   U_r=X_r-\frac{\lambda}{4},
   \qquad
   V_r=Y_r-\frac{\lambda}{4},
   \qquad r\in\mathbb Z_q.
\]
Then $U_{r-1}+U_r+U_{r+1}+V_r=0$
and $U_r+V_{r-k}+V_r+V_{r+k}=0$,
for every \(r\in\mathbb Z_q\). By Lemma~\ref{lem:circulant-system}, with
\(\alpha=1\), and Lemma~\ref{lem:root-nonzero}\((1)\), we have $U_r=V_r=0$, for all $r\in\mathbb Z_q$.
Hence $X_r=Y_r=\frac{\lambda}{4}$,
for all $r\in\mathbb Z_q$. For \((2)\), put
\[
   U_r=X_r-\frac{\lambda}{5},
   \qquad
   V_r=Y_r-\frac{\lambda}{5},
   \qquad r\in\mathbb Z_q.
\]
Then
$U_{r-1}+2U_r+U_{r+1}+V_r=0$
and
$U_r+V_{r-k}+2V_r+V_{r+k}=0$,
for every \(r\in\mathbb Z_q\). By Lemma~\ref{lem:circulant-system}, with
\(\alpha=2\), and Lemma~\ref{lem:root-nonzero}\((2)\), we get
$U_r=V_r=0$ for every $r\in\mathbb Z_q$. Thus
$X_r=Y_r=\frac{\lambda}{5}$ for every $r\in\mathbb Z_q$. For \((3)\), put
\[
   U_r=S_r-\frac{\lambda}{3},
   \qquad r\in\mathbb Z_q.
\]
The assumed equation becomes
\[
   U_{r-k-1}+U_{r-k+1}+U_{r+k-1}+U_{r+k+1}=U_r,
   \qquad r\in\mathbb Z_q.
\]
Define $V_r:=-(U_{r-1}+U_{r+1})$ for every $r\in\mathbb Z_q$. Then
$U_{r-1}+U_{r+1}+V_r=0$ for every $r\in\mathbb Z_q$.
Moreover,
\[
\begin{aligned}
   U_r+V_{r-k}+V_{r+k}
   &=
   U_r-(U_{r-k-1}+U_{r-k+1})
       -(U_{r+k-1}+U_{r+k+1})  \\
   &=0.
\end{aligned}
\]
Therefore the sequences \((U_r)\) and \((V_r)\) satisfy the system of
Lemma~\ref{lem:circulant-system} with \(\alpha=0\). By
Lemma~\ref{lem:root-nonzero}\((3)\), the determinant condition in
Lemma~\ref{lem:circulant-system} is satisfied. Hence
\[
   U_r=V_r=0,
   \qquad r\in\mathbb Z_q.
\]
In particular,
\[
   S_r=\frac{\lambda}{3},
   \qquad r\in\mathbb Z_q.
\]

\noindent The last assertion follows immediately from the displayed equalities.
\end{proof}

\section{The quotient matrix \(A_2\)}
\label{sec:A2}

In this section we determine when the matrix
\[
   A_2=
   \begin{pmatrix}
      2 & 1\\
      2 & 1
   \end{pmatrix}
\]
occurs as the quotient matrix of a perfect \(2\)-coloring of \(GP(n,k)\).

\begin{theorem}\label{thm:A2}
Let \(n\geq 3\) and \(1\leq k<n/2\). Then \(GP(n,k)\) has a perfect
\(2\)-coloring with quotient matrix
\[
   A_2=
   \begin{pmatrix}
      2 & 1\\
      2 & 1
   \end{pmatrix}
\]
if and only if
\[
   3\mid n
   \qquad\text{and}\qquad
   3\nmid k .
\]
\end{theorem}

\begin{proof}
The existence for \(3\mid n\) and \(3\nmid k\) follows from the periodic
construction in \cite[Theorem~3.3]{AlaeiyanKarami2016GP}. We prove the
converse.

\noindent Suppose that \(T\) is a perfect \(2\)-coloring of \(GP(n,k)\) with quotient
matrix \(A_2\). Let
\[
   C_i=\{v\in V(GP(n,k)) : T(v)=i\},
   \qquad i=1,2.
\]
Counting the edges between \(C_1\) and \(C_2\) in two ways gives
\[
   |C_1|=2|C_2|.
\]
Since \(|C_1|+|C_2|=2n\), we obtain
\[
   |C_2|=\frac{2n}{3}.
\]
Hence \(3\mid n\).

It remains to show that \(3\nmid k\). Suppose, to the contrary, that
\(3\mid k\). Write
\[
   n=3^s m,
   \qquad s\geq 1,
   \qquad 3\nmid m,
\]
and put \(q=3^s\). Let
\[
   \chi(v)=
   \begin{cases}
      1, & T(v)=2,\\
      0, & T(v)=1.
   \end{cases}
\]
Since every vertex has exactly one neighbour of color \(2\), we have, for every
\(i\in \mathbb Z_n\),
\[
   \chi(a_{i-1})+\chi(a_{i+1})+\chi(b_i)=1
\]
and
\[
   \chi(a_i)+\chi(b_{i-k})+\chi(b_{i+k})=1.
\]
From the first equation,
\[
   \chi(b_i)=1-\chi(a_{i-1})-\chi(a_{i+1}).
\]
Substituting this into the second equation gives
\[
   \chi(a_{i-k-1})+\chi(a_{i-k+1})
   +\chi(a_{i+k-1})+\chi(a_{i+k+1})
   =
   \chi(a_i)+1
\]
for all \(i\in\mathbb Z_n\). For \(r\in\mathbb Z_q\), define
\[
   S_r=\sum_{j=0}^{m-1}\chi(a_{r+jq}).
\]
Summing the last recurrence over all indices \(i=r+jq\), \(j=0,1,\ldots,m-1\),
we obtain
\[
   S_{r-k-1}+S_{r-k+1}+S_{r+k-1}+S_{r+k+1}=S_r+m,
   \qquad r\in\mathbb Z_q.
\]
Here the subscripts are read modulo \(q\). Since \(q=3^s\) and \(3\mid k\),
Corollary~\ref{cor:cyclic-uniformity}\((3)\), with \(\lambda=m\), gives
\[
   S_r=\frac{m}{3},
   \qquad r\in\mathbb Z_q.
\]
This is impossible, because each \(S_r\) is an integer while \(3\nmid m\).
Therefore \(3\nmid k\).
\end{proof}

\section{The quotient matrix \(A_3\)}
\label{sec:A3}

The case \(k=3\) was treated in \cite[Theorem~3.3]{Karami2022GPn3}. We prove
the general statement.

\begin{theorem}\label{thm:A3}
Let \(n\geq 3\) and \(1\leq k<n/2\). Then \(GP(n,k)\) has a perfect
\(2\)-coloring with quotient matrix
\[
   A_3=
   \begin{pmatrix}
      1 & 2\\
      2 & 1
   \end{pmatrix}
\]
if and only if
\[
   2\mid n
   \qquad\text{and}\qquad
   k \text{ is odd}.
\]
\end{theorem}

\begin{proof}
Suppose first that \(2\mid n\) and \(k\) is odd. Define
\(T:V(GP(n,k))\to\{1,2\}\) by
\[
   T(a_{2i})=T(b_{2i})=1,
   \qquad
   T(a_{2i+1})=T(b_{2i+1})=2 .
\]
This coloring is well-defined because \(n\) is even. Since \(k\) is odd, the
vertices \(b_i\) and \(b_{i+k}\) have opposite colors. It follows immediately
that every vertex has exactly one neighbour of its own color and two neighbours
of the other color. Thus \(T\) is a perfect \(2\)-coloring with quotient matrix
\(A_3\).

\noindent Conversely, suppose that \(T\) is a perfect \(2\)-coloring of \(GP(n,k)\) with
quotient matrix \(A_3\). Put
\[
   x_i=
   \begin{cases}
      1, & T(a_i)=1,\\
      0, & T(a_i)=2,
   \end{cases}
   \qquad
   y_i=
   \begin{cases}
      1, & T(b_i)=1,\\
      0, & T(b_i)=2.
   \end{cases}
\]
All indices are read modulo \(n\). For the matrix \(A_3\), every closed neighbourhood contains exactly two
vertices of color \(1\). Hence, for every \(i\in\mathbb Z_n\),
\[
   x_{i-1}+x_i+x_{i+1}+y_i=2
\]
and
\[
   x_i+y_{i-k}+y_i+y_{i+k}=2.
\]
Summing these two equations over all \(i\in\mathbb Z_n\), we obtain
\[
   3\sum_{i\in\mathbb Z_n}x_i+\sum_{i\in\mathbb Z_n}y_i=2n
\]
and
\[
   \sum_{i\in\mathbb Z_n}x_i+3\sum_{i\in\mathbb Z_n}y_i=2n.
\]
Therefore
\[
   \sum_{i\in\mathbb Z_n}x_i=
   \sum_{i\in\mathbb Z_n}y_i=
   \frac n2.
\]
In particular, \(n\) is even. It remains to prove that \(k\) is odd. Suppose, to the contrary, that \(k\) is
even. Write
\[
   n=2^s m,
   \qquad s\geq 1,
   \qquad m \text{ odd},
\]
and put \(q=2^s\). For \(r\in\mathbb Z_q\), define
\[
   X_r=\sum_{j=0}^{m-1}x_{r+jq},
   \qquad
   Y_r=\sum_{j=0}^{m-1}y_{r+jq}.
\]
Summing the two displayed equations over all indices \(i=r+jq\), we get
\[
   X_{r-1}+X_r+X_{r+1}+Y_r=2m,
   \qquad r\in\mathbb Z_q,
\]
and
\[
   X_r+Y_{r-k}+Y_r+Y_{r+k}=2m,
   \qquad r\in\mathbb Z_q.
\]
By Corollary~\ref{cor:cyclic-uniformity}\((1)\), with \(\lambda=2m\),
\[
   X_r=Y_r=\frac m2,
   \qquad r\in\mathbb Z_q.
\]
This is impossible, since \(X_r\) and \(Y_r\) are integers and \(m\) is odd.
Thus \(k\) is odd.
\end{proof}
\section{The quotient matrix \(A_4\)}
\label{sec:A4}

For \(k=3\), this case was obtained in \cite[Theorem~3.4]{Karami2022GPn3}.
The following proof gives the general case.

\begin{theorem}\label{thm:A4}
Let \(n\geq 3\) and \(1\leq k<n/2\). Then \(GP(n,k)\) has a perfect
\(2\)-coloring with quotient matrix
\[
   A_4=
   \begin{pmatrix}
      0 & 3\\
      1 & 2
   \end{pmatrix}
\]
if and only if $4\mid n$ and \(k\) is odd.
\end{theorem}

\begin{proof}
Suppose first that \(4\mid n\) and \(k\) is odd. Write \(n=4m\). Define
\(T:V(GP(n,k))\to\{1,2\}\) by
\[
   T(a_{4i})=T(b_{4i+2})=1,
   \qquad i=0,1,\ldots,m-1,
\]
and assign color \(2\) to all remaining vertices. We check the coloring modulo \(4\). For a vertex \(a_i\), the closed
neighbourhood is
\[
   \{a_{i-1},a_i,a_{i+1},b_i\}.
\]
It contains exactly one vertex of color \(1\): namely \(a_i\) if
\(i\equiv 0\pmod 4\), \(a_{i-1}\) if \(i\equiv 1\pmod 4\), \(b_i\) if
\(i\equiv 2\pmod 4\), and \(a_{i+1}\) if \(i\equiv 3\pmod 4\). For a vertex \(b_i\), the closed neighbourhood is
\[
   \{a_i,b_{i-k},b_i,b_{i+k}\}.
\]
Since \(k\) is odd, a direct check modulo \(4\) again shows that this set
contains exactly one vertex of color \(1\). Hence every vertex of color \(1\)
has no neighbour of color \(1\), and every vertex of color \(2\) has exactly one
neighbour of color \(1\). Thus \(T\) is a perfect \(2\)-coloring with quotient
matrix \(A_4\).

\noindent Conversely, suppose that \(T\) is a perfect \(2\)-coloring of \(GP(n,k)\) with
quotient matrix \(A_4\). Put
\[
   x_i=
   \begin{cases}
      1, & T(a_i)=1,\\
      0, & T(a_i)=2,
   \end{cases}
   \qquad
   y_i=
   \begin{cases}
      1, & T(b_i)=1,\\
      0, & T(b_i)=2.
   \end{cases}
\]
All indices are read modulo \(n\). For the matrix \(A_4\), every closed neighbourhood contains exactly one vertex
of color \(1\). Thus, for every \(i\in\mathbb Z_n\),
\[
   x_{i-1}+x_i+x_{i+1}+y_i=1
\]
and
\[
   x_i+y_{i-k}+y_i+y_{i+k}=1.
\]
Summing over all \(i\in\mathbb Z_n\), we get
\[
   3\sum_{i\in\mathbb Z_n}x_i+\sum_{i\in\mathbb Z_n}y_i=n
\]
and
\[
   \sum_{i\in\mathbb Z_n}x_i+3\sum_{i\in\mathbb Z_n}y_i=n.
\]
Therefore
\[
   \sum_{i\in\mathbb Z_n}x_i=
   \sum_{i\in\mathbb Z_n}y_i=
   \frac n4.
\]
Hence \(4\mid n\).

\noindent It remains to prove that \(k\) is odd. Suppose, to the contrary, that \(k\) is
even. Write
\[
   n=2^s m,
   \qquad s\geq 2,
   \qquad m \text{ odd},
\]
and put \(q=2^s\). For \(r\in\mathbb Z_q\), define
\[
   X_r=\sum_{j=0}^{m-1}x_{r+jq},
   \qquad
   Y_r=\sum_{j=0}^{m-1}y_{r+jq}.
\]
Summing the two displayed equations over all indices \(i=r+jq\), we obtain
\[
   X_{r-1}+X_r+X_{r+1}+Y_r=m,
   \qquad r\in\mathbb Z_q,
\]
and
\[
   X_r+Y_{r-k}+Y_r+Y_{r+k}=m,
   \qquad r\in\mathbb Z_q.
\]
By Corollary~\ref{cor:cyclic-uniformity}\((1)\), with \(\lambda=m\),
\[
   X_r=Y_r=\frac m4,
   \qquad r\in\mathbb Z_q.
\]
This is impossible, since \(X_r\) and \(Y_r\) are integers and \(m\) is odd.
Thus \(k\) is odd.
\end{proof}
\section{The quotient matrix \(A_5\)}
\label{sec:A5}

The period-five construction used below is the same construction that appears in
\cite[Theorem~3.5]{Karami2022GPn3}. We include it for completeness.

\begin{theorem}\label{thm:A5}
Let \(n\geq 3\) and \(1\leq k<n/2\). Then \(GP(n,k)\) has a perfect
\(2\)-coloring with quotient matrix
\[
   A_5=
   \begin{pmatrix}
      0 & 3\\
      2 & 1
   \end{pmatrix}
\]
if and only if $5\mid n$ and $k\equiv \pm2 \pmod 5$.
\end{theorem}

\begin{proof}
Suppose first that \(5\mid n\) and \(k\equiv \pm2\pmod 5\). Write \(n=5m\).
Define \(T:V(GP(n,k))\to\{1,2\}\) by
\[
   T(a_{5i+1})=T(a_{5i+4})=T(b_{5i+2})=T(b_{5i+3})=1,
   \qquad i=0,1,\ldots,m-1,
\]
and assign color \(2\) to all remaining vertices. We check the coloring modulo \(5\). For an outer vertex \(a_i\), the vertices
\(a_{i-1},a_{i+1},b_i\) contain no vertex of color \(1\) when
\(i\equiv 1,4\pmod 5\), and contain exactly two vertices of color \(1\)
otherwise. For an inner vertex \(b_i\), since \(k\equiv\pm2\pmod 5\), the
vertices \(a_i,b_{i-k},b_{i+k}\) contain no vertex of color \(1\) when
\(i\equiv 2,3\pmod 5\), and contain exactly two vertices of color \(1\)
otherwise. Hence every vertex of color \(1\) has no neighbour of color \(1\),
and every vertex of color \(2\) has exactly two neighbours of color \(1\). Thus
\(T\) is a perfect \(2\)-coloring with quotient matrix \(A_5\).

\noindent Conversely, suppose that \(T\) is a perfect \(2\)-coloring of \(GP(n,k)\) with
quotient matrix \(A_5\). Put
\[
   x_i=
   \begin{cases}
      1, & T(a_i)=1,\\
      0, & T(a_i)=2,
   \end{cases}
   \qquad
   y_i=
   \begin{cases}
      1, & T(b_i)=1,\\
      0, & T(b_i)=2.
   \end{cases}
\]
All indices are read modulo \(n\). For the matrix \(A_5\), a vertex of color \(1\) has no neighbour of color \(1\),
whereas a vertex of color \(2\) has exactly two neighbours of color \(1\). Hence,
for every \(i\in\mathbb Z_n\),
\[
   x_{i-1}+2x_i+x_{i+1}+y_i=2
\]
and
\[
   x_i+y_{i-k}+2y_i+y_{i+k}=2.
\]
Summing over all \(i\in\mathbb Z_n\), we obtain
\[
   4\sum_{i\in\mathbb Z_n}x_i+\sum_{i\in\mathbb Z_n}y_i=2n
\]
and
\[
   \sum_{i\in\mathbb Z_n}x_i+4\sum_{i\in\mathbb Z_n}y_i=2n.
\]
Therefore
\[
   \sum_{i\in\mathbb Z_n}x_i=
   \sum_{i\in\mathbb Z_n}y_i=
   \frac{2n}{5}.
\]
Thus \(5\mid n\). It remains to exclude the congruence classes \(k\equiv 0,\pm1\pmod 5\).
Suppose, to the contrary, that
\[
   k\equiv 0,\pm1\pmod 5.
\]
Write
\[
   n=5^s m,
   \qquad s\geq 1,
   \qquad 5\nmid m,
\]
and put \(q=5^s\). For \(r\in\mathbb Z_q\), define
\[
   X_r=\sum_{j=0}^{m-1}x_{r+jq},
   \qquad
   Y_r=\sum_{j=0}^{m-1}y_{r+jq}.
\]
Summing the two displayed equations over all indices \(i=r+jq\), we get
\[
   X_{r-1}+2X_r+X_{r+1}+Y_r=2m,
   \qquad r\in\mathbb Z_q,
\]
and
\[
   X_r+Y_{r-k}+2Y_r+Y_{r+k}=2m,
   \qquad r\in\mathbb Z_q.
\]
By Corollary~\ref{cor:cyclic-uniformity}\((2)\), with \(\lambda=2m\),
\[
   X_r=Y_r=\frac{2m}{5},
   \qquad r\in\mathbb Z_q.
\]
This is impossible, since \(X_r\) and \(Y_r\) are integers and \(5\nmid m\).
Therefore \(k\not\equiv 0,\pm1\pmod 5\), and hence
\[
   k\equiv \pm2\pmod 5.
\]
\end{proof}

\begin{proof}[Proof of Theorem~\ref{thm:main-classification}]
The case \(A_1\) is obtained by coloring all vertices \(a_i\) with one color
and all vertices \(b_i\) with the other color. The case \(A_6\) is equivalent
to bipartiteness of \(GP(n,k)\), which holds if and only if \(n\) is even and
\(k\) is odd. The remaining cases \(A_2,A_3,A_4\), and \(A_5\) are exactly
Theorems~\ref{thm:A2}, \ref{thm:A3}, \ref{thm:A4}, and \ref{thm:A5}.
\end{proof}

\section{Concluding remark}

The classification above shows that, for generalized Petersen graphs, the
existence of completely regular codes of covering radius one is governed only
by simple congruence conditions on \(n\) and \(k\). The nonexistence part is
reduced to uniformity of cyclic systems, while all feasible cases are realized
by periodic colorings.

\section*{Declarations}

\textbf{Funding.}
The author received no funding for this work.

\noindent \textbf{Competing interests.}
The author declares that there are no competing interests.

\noindent\textbf{Author contributions.}
The author is solely responsible for all aspects of the work.

\noindent\textbf{Data availability.}
Data sharing is not applicable to this article as no datasets were generated or analyzed during the current study.
\bibliographystyle{plain}
\bibliography{mybib}
\end{document}